\documentclass[12pt]{amsart}
\advance\hoffset by -0.5 truecm
\usepackage{amssymb}
\newtheorem{Theorem}{Theorem}[section]
\newtheorem{Lemma}[Theorem]{Lemma}

\newtheorem{Proposition}[Theorem]{Proposition}

\newtheorem{Definition}[Theorem]{Definition}
\newtheorem{Remark}[Theorem]{Remark}

\def\V{\mbox{Var}}

\def\R\re
\def\V{\bf V}

\def \re{{\mathbb R}}

\def \C{{\mathbb C}}

\def \V{{\bf V}}

\def \X{X_{\perp}}
\def \H{\mathcal H}
\newcommand{\id}{\mathrm{Id}}
\newcommand{\hilb}{\mathcal H}

\newcommand{\abs}[1]{\lvert #1 \rvert}

\begin{document}
\title[The attenuated ray transform for connections]{The attenuated ray transform for connections and Higgs fields}

\author[G.P. Paternain]{Gabriel P. Paternain}
 \address{ Department of Pure Mathematics and Mathematical Statistics,
University of Cambridge,
Cambridge CB3 0WB, UK}
 \email {g.p.paternain@dpmms.cam.ac.uk}

\author[M. Salo]{Mikko Salo}
\address{Department of Mathematics and Statistics, University of Helsinki and University of Jyv\"askyl\"a}
\email{mikko.salo@helsinki.fi}

\author[G. Uhlmann]{Gunther Uhlmann}
\address{Department of Mathematics, University of Washington and University of California, Irvine}

\email{gunther@math.washington.edu}




\begin{abstract} 
We show that for a simple surface with boundary the attenuated ray transform in the presence of a unitary connection and a skew-Hermitian Higgs field is injective modulo the natural obstruction for functions and vector fields. We also show that the connection and the Higgs field are uniquely determined by the scattering relation modulo a gauge transformation. The proofs involve a Pestov type energy identity for connections together with holomorphic gauge transformations which arrange the curvature of the connection to have definite sign.
\end{abstract}

\maketitle

\section{Introduction}

In this paper we consider a generalization of the attenuated ray
transform in two dimensions to the case when the attenuation is given by a connection and a
Higgs field. Before describing the problem in more detail we recall
the case of the attenuated ray transform on manifolds.

Let $(M,g)$ be a compact oriented Riemannian manifold with smooth boundary, and let $SM = \{(x,v) \in TM \,;\, \abs{v} = 1 \}$ be the unit tangent bundle. The geodesics going from $\partial M$ into $M$ can be parametrized by the set $\partial_+ (SM) = \{(x,v) \in SM \,;\, x \in \partial M, \langle v,\nu \rangle \leq 0 \}$ where $\nu$ is the outer unit normal vector to $\partial M$. For any $(x,v) \in SM$ we let $t \mapsto \gamma(t,x,v)$ be the geodesic starting from $x$ in direction $v$. We assume that $(M,g)$ is nontrapping, which means that the time $\tau(x,v)$ when the geodesic $\gamma(t,x,v)$ exits $M$ is finite for each $(x,v) \in SM$.

If $a \in C^{\infty}(M)$ is the attenuation coefficient, consider the attenuated ray transform of a function $f \in C^{\infty}(M)$, defined for $(x,v) \in \partial_+ (SM)$ by 
\begin{equation*}
I_a f(x,v) = \int_0^{\tau(x,v)} f(\gamma(t,x,v)) \,\text{exp} \Big[ \int_0^{t} a(\gamma(s,x,v)) \,ds \Big] \,dt.
\end{equation*}

In \cite{SaU} it is shown that the attenuated ray transform is injective for simple two dimensional manifolds. Moreover in this article stability estimates and a reconstruction procedure of the function from the attenuated transform are given. Also the case of integrating along one forms was considered.
A compact Riemannian manifold with boundary is said to be {\it simple} if for any point $p \in M$   the exponential map $\exp_p$   is a diffeomorphism onto $M$, and if the boundary is strictly convex. The notion of simplicity arises naturally in the context of the boundary rigidity problem \cite{Mi}.
In the case of Euclidean space with the Euclidean metric the attenuated ray transform is the basis of the medical imaging technology of SPECT and has been extensively studied, see 
\cite{finch} for a review and also the remarks after Theorem \ref{thm:injective_higgs}.

When studying ray transforms it is very useful to consider them as boundary values of solutions of transport equations. For the case considered above
the appropriate transport equation is  
$$
Xu + au = -f \ \ \text{in $SM$}, \quad u|_{\partial_-(SM)} = 0.
$$
Here $X$ is the geodesic vector field and  $\partial_- (SM) = \{(x,v) \in SM \,;\, x \in \partial M, \langle v,\nu \rangle \geq 0 \}$. The attenuated ray transform is then given by 
$$I_a f := u|_{\partial_+ (SM)}.$$
 
In this paper we will consider a generalization of the above setup to systems, where instead of a scalar function $a$ on $M$ the attenuation is given by a connection and a Higgs field on the trivial bundle $M \times \C^n$. For us a connection $A$ will be a complex $n\times n$ matrix whose
entries are smooth 1-forms on $M$. Another way to think of $A$ is to regard
it as a smooth map $A:TM\to \C^{n\times n}$ which is linear in $v\in T_{x}M$ for
each $x\in M$. Very often in physics and geometry one considers {\it unitary} or {\it Hermitian} connections. This means that the range of $A$ is restricted to skew-Hermitian matrices. In other words, if we denote by $\mathfrak{u}(n)$ the Lie algebra of the unitary group $U(n)$, we have a smooth map
$A:TM\to \mathfrak{u}(n)$ which is linear in the velocities.
There is yet another equivalent way to phrase this. The connection $A$ induces
a covariant derivative $d_{A}$ on sections $s\in C^{\infty}(M,\C^n)$ by setting
$d_{A}s=ds+As$. Then $A$ being Hermitian or unitary is equivalent to requiring compatibility with the standard Hermitian inner product of $\C^n$ in the sense that
\[d\langle s_{1},s_{2}\rangle=\langle d_{A}s_{1},s_{2}\rangle+\langle s_{1},d_{A}s_{2}\rangle\]
for any pair of functions $s_{1},s_{2}$.
Another geometric object which is often introduced in physics is a {\it Higgs
field}. For us, this means a smooth matrix-valued function
$\Phi:M\to \C^{n\times n}$. Often in gauge theories, the structure group
is $U(n)$ and the field $\Phi$ is required to take values in $\mathfrak{u}(n)$.
We call a Higgs field $\Phi:M\to\mathfrak{u}(n)$ a skew-Hermitian Higgs field.

The pairs $(A,\Phi)$ often appear in the so-called Yang-Mills-Higgs theories
and the most popular and important structure groups are $U(n)$ or $SU(n)$.
A good example of this is the Bogomolny equation in Minkowski $(2+1)$-space given
by $d_{A}\Phi=\star F_{A}$. Here $d_{A}$ stands for the covariant derivative
induced on endomorphism $d_{A}\Phi=d\Phi+[A,\Phi]$, $F_{A}=dA+A\wedge A$ is the curvature of $A$ and $\star$ is the Hodge star operator of Minkowski space.
The Bogomolny equation appears as a reduction of the self-dual Yang-Mills equation in $(2+2)$-space and has been object of intense study in the literature
of Solitons and Integrable Systems, see for instance \cite{D}, \cite[Chapter 8]{MS}, \cite[Chapter 4]{HSW} and \cite{MW}.

Consider the following transport equation for $u: SM \to \C^n$,
$$Xu + Au + \Phi u = -f \ \ \text{in $SM$}, \quad u|_{\partial_-(SM)} = 0.
$$ Here $A(x,v)$ (the restriction of $A$ to $SM$) and $\Phi(x)$ act on functions on $SM$ by multiplication. On a fixed geodesic the transport equation becomes a linear system of ODEs with zero initial condition, and therefore this equation has a unique solution $u=u^f$. 

\begin{Definition}The geodesic ray transform of $f \in C^{\infty}(SM,\C^n)$ with attenuation determined by the pair $(A,\Phi)$ is given by$$I_{A,\Phi} f := u^f|_{\partial_+(SM)}.$$ 
When $\Phi=0$ we abbreviate $I_{A}:=I_{A,0}$.
\end{Definition}

We state our first main result for the case of functions independent of the fiber and Higgs field
equal to zero.

\begin{Theorem} Let $M$ be a compact simple surface. Assume that $f:M\to\C^n$ is a smooth function and let $A: TM \to \mathfrak{u}(n)$ be a unitary connection. If $I_{A}(f)=0$, then $f=0.$
\end{Theorem}

In fact the last result is a corollary of the following theorem, which also includes a skew-Hermitian Higgs field and applies to functions $f$ which are sums of $0$-forms and $1$-forms.

\begin{Theorem} Let $M$ be a compact simple surface. Assume that $f:SM\to\C^n$ is a smooth function of the form
$F(x)+\alpha_{j}(x)v^j$, where $F:M\to\C^n$ is a smooth function
and $\alpha$ is a $\C^n$-valued 1-form. Let also $A: TM \to \mathfrak{u}(n)$ be a unitary connection and $\Phi: M \to \mathfrak{u}(n)$ a skew-Hermitian matrix function. If $I_{A,\Phi}(f)=0$, then
$F = \Phi p$ and $\alpha=d_{A}p$, where $p:M\to\C^n$ is a smooth
function with $p|_{\partial M}=0$.
\label{thm:injective_higgs}
\end{Theorem}

We remark that in connection with injectivity results for ray transforms, there is great interest in reconstruction procedures and inversion formulas. For the attenuated ray transform in $\re^2$ with Euclidean metric and scalar attenuation function, an explicit inversion formula was proved by R. Novikov \cite{No_inversion}. A related formula also including $1$-form attenuations appears in \cite{BS}, inversion formulas for matrix attenuations in Euclidean space are given in \cite{E, No}, and the case of hyperbolic space ${\mathbb H}^2$ is considered in \cite{Bal}. We are not aware of explicit inversion formulas in geometries with nonconstant curvature in the literature; even with zero attenuation, only Fredholm type inversion formulas on simple surfaces are available in general \cite{PU0} (see however the recent work \cite{BalHoell}). A possible reconstruction procedure for the scalar attenuated ray transform on simple surfaces was discussed in \cite{SaU}. The methods in the current paper suggest a similar procedure for systems, except that the current proof of Theorem \ref{thm:pi} is not constructive.

Let us move to the final main result of this paper, concerning an inverse problem with scattering data. On a nontrapping compact manifold $(M,g)$ with strictly convex boundary, the scattering relation $\alpha = \alpha_g: \partial_+(SM) \to \partial_-(SM)$ maps a starting point and direction of a geodesic to the end point and direction. If $(M,g)$ is simple, then knowing $\alpha_g$ is equivalent to knowing the boundary distance function $d_g$ which encodes the distances between any pair of boundary points \cite{Mi}. On two dimensional simple manifolds, the boundary distance function $d_g$ determines the metric $g$ up to an isometry which fixes the boundary \cite{PU}.

Given a connection and a Higgs field on the bundle $M \times \C^n$, there is an additional piece of scattering data. Consider the unique matrix solution
$U_{-}:SM\to GL(n,\C)$ to the transport equation
$$
X U_{-} + A U_- + \Phi U_{-} = 0 \text{ in } SM, \quad U_{-}|_{\partial_{+}(SM)}=\id.
$$

\begin{Definition}
The scattering data (or relation) corresponding to a matrix attenuation pair $(A,\Phi)$ in $(M,g)$ is the map
$$
C^{A,\Phi}_{-}:\partial_{-}(SM)\to GL(n,\C), \ \  C^{A,\Phi}_{-} := U_{-}|_{\partial_{-}(SM)}.
$$
\end{Definition}

The scattering relation for the case of a connection with Higgs field is gauge invariant. 
Let $Q: M \to GL(n,\C)$ be a smooth function taking values in invertible matrices, then $Q^{-1} U_{-}$ satisfies
$$(X+Q^{-1}(X+A)Q + Q^{-1} \Phi Q)(Q^{-1}U_-) = 0 \ \ \text{in $SM$}, \ \ \ Q^{-1} U_-|_{\partial_+(SM)} = Q^{-1}|_{\partial_+(SM)}.
$$
The scattering data has therefore the gauge invariance
$$C^{Q^{-1}(X+A)Q,Q^{-1} \Phi Q}_- = C^{A,\Phi}_- \quad \text{if $Q \in C^{\infty}(M,GL(n,\C))$ satisfies $Q|_{\partial M} = \id$}.
$$It follows that from the knowledge of $C^{A,\Phi}_{-}$ one can only expect to recover $A$ and $\Phi$ up to a gauge transformation via $Q$ which satisfies $Q|_{\partial M} = \id$. If $A$ is unitary and $\Phi$ is skew-Hermitian, the map $U_{-}$ and the scattering relation
$C^{A,\Phi}_{-}$ take values in the unitary group $U(n)$ and the scattering relation remains unchanged under unitary gauge transformations which
are the identity on the boundary.

In Section \ref{sec:scattering_relation}, we prove that in the unitary case the scattering relation determines the pair $(A,\Phi)$ up to the natural gauge equivalence; the proof uses the injectivity result for the attenuated ray transform given in Theorem \ref{thm:injective_higgs}.

\begin{Theorem} Assume $M$ is a compact simple surface, let
$A$ and $B$ be two Hermitian connections, and let $\Phi$ and $\Psi$ be two skew-Hermitian Higgs fields.
Then $C^{A,\Phi}_{-}=C^{B,\Psi}_{-}$ implies that there exists a smooth
$U:M\to U(n)$ such that $U|_{\partial M}=\id$ and
$B=U^{-1}dU+U^{-1}AU$, $\Psi = U^{-1} \Phi U$.
\label{thm:inverse}
\end{Theorem}

Various versions of Theorem \ref{thm:inverse} have
been proved in the literature, often for the case $\Phi = 0$. Sharafutdinov \cite{Sha} proves the theorem
assuming that the connections are $C^1$ close to another connection with small curvature (but in any dimension).
In the case of domains in the Euclidean plane the theorem was proved by Finch and Uhlmann \cite{FU}
assuming that the connections have small curvature and by
G. Eskin \cite{E} in general.
R. Novikov \cite{No} considers the case of connections which
are not compactly supported (but with suitable decay conditions at infinity) and establishes local uniqueness of the trivial connection and gives examples in which global uniqueness fails (existence of ``ghosts"). Global uniqueness also fails in the case
of closed manifolds and a complete description of the ghosts (``transparent connections and pairs'') appears in \cite{P,P1,P2} for the case of negatively curved surfaces and structure group $SU(2)$. A full classification of $U(n)$ transparent connections for the round metric
on $S^2$ has been obtained by L. Mason (unpublished) using methods from twistor theory as in \cite{Ma}.

There is a remarkable connection between 
the Bogomolny equation and the scattering data associated with the
transport equation considered in this paper. As it is explained in \cite{W}
(see also \cite[Section 8.2.1]{D}), certain soliton pairs $(A,\Phi)$
have the property that when restricted to space-like planes the scattering data is trivial. These are precisely the examples
considered by R. Novikov \cite{No} and mentioned above. In this way one obtains connections and Higgs fields in $\mathbb R^2$ with the property
of having trivial scattering data but which are not gauge equivalent to the trivial pair.
Of course these pairs are not compactly supported in $\mathbb R^2$ but they have a suitable decay at infinity.   

We also point out that
the problem of determining a connection from the scattering relation arises naturally when considering the hyperbolic Dirichlet-to-Neumann map associated to the Schr\"odinger equation with a connection. It was shown in \cite{FU} that when the metric is Euclidean, the scattering relation for a connection $A$ (with zero Higgs field) can be determined from the hyperbolic Dirichlet-to-Neumann map. A similar result holds true on Riemannian manifolds: a combination of the methods in \cite{FU} and \cite{U2} shows that the hyperbolic Dirichlet-to-Neumann map for a connection determines the scattering data $C^{A,0}_-$. We do not consider the details here.

A brief summary of the rest of the paper is as follows: in Section \ref{sec:prelim} we review some preliminary material and we establish a commmutator formula between the fibrewise Hilbert transform and the operator $X+A+\Phi$. In Section \ref{scatt} we discuss in more detail the attenuated ray transform, we define the scattering data and show its gauge equivalence. In Section \ref{sec:hol} we consider the scalar case.
We prove that the transport equation in this case has holomorphic and antiholomorphic (in the fiber variable) integrating factors. A special case of this result was considered in \cite{SaU}.  This result will be used in the subsequent sections. In Section \ref{sec:regularity} we study the regularity of solutions of the transport equation. In Section \ref{sec:injective} we show injectivity of the attenuated transform on simple surfaces when the Higgs field vanishes using (what appears to be) a new Pestov type identity (or Energy identity). Using the existence of holomorphic integrating factors we are able to modify the connection to control the relevant curvature term appearing
in the Pestov type identity. The ideas employed here resemble the familiar techniques in Complex Geometry that establish results
like the Kodaira vanishing theorem.
In Section \ref{sec:injective_higgs} we prove injectivity in the general case when there is a Higgs field present. This requires an additional analysis since new terms need to be controlled in the Pestov identity. Finally, in Section \ref{sec:scattering_relation} we prove that the scattering relation determines the connection and the Higgs field modulo a gauge transformation.

\subsection*{Acknowledgements}

M.S. was supported in part by the Academy of Finland, and G.U. was partly supported by NSF, a Senior Clay Award and a Chancellor Professorship at UC Berkeley. This work was initiated during the MSRI program on Inverse problems and applications, and the authors would like to express their gratitude to MSRI and the organizers of the program. The last two authors would also like to thank the organizers of the Trimester on inverse problems in Madrid (project CSD2006-00032 on Inverse problems and scattering), where part of the work was done. Finally, the authors would like to thank Colin Guillarmou and Leo Tzou for several helpful discussions on the topic.

\section{Preliminaries} \label{sec:prelim}

Let $(M,g)$ be a compact oriented two dimensional Riemannian manifold with smooth boundary
$\partial M$. As usual $SM$ will denote the unit circle bundle which is a compact 3-manifold with boundary given by $\partial(SM)=\{(x,v)\in SM:\;x\in \partial M\}$. The outer unit normal of $\partial M$ is denoted by $\nu$ and the sets of inner and outer vectors on $\partial M$ are given by
\[\partial_{\pm}(SM)=\{(x,v)\in SM:\;x\in\partial M,\;\pm\langle v,\nu\rangle\leq 0\}.\]
The nonnegative time when the geodesic determined by $(x,v)$ exits $M$ is denoted by $\tau(x,v)$. The manifold is said to be \emph{nontrapping} if $\tau(x,v)$ is finite for any $(x,v)\in SM$. The boundary of $M$ is said to be \emph{strictly convex} if its second fundamental form is positive definite. Recall that given $x\in \partial M$, the second fundamental form of the boundary at $x$ is the quadratic form $\Pi_{x}:T_{x}\partial M\times T_{x}\partial M\to \mathbb {R}$ defined by $\Pi_{x}(w,w)=\langle \nabla_{w}\nu,w\rangle$. Hence we require that $\Pi_{x}$ is positive
definite for all $x\in\partial M$.
 The manifold is said to be \emph{simple} if it is compact with smooth strictly convex boundary, and the exponential map at any point is a diffeomorphism onto $M$.

If $(M,g)$ is nontrapping and has strictly convex boundary then $\tau$ is continuous on $SM$ and smooth in $SM\setminus S(\partial M)$, see \cite[Section 4.1]{Sh}. Moreover the function $\tau_{-}:\partial(SM)\to\mathbb R$  defined by
\[\tau_{-}(x,v)=\left\{\begin{array}{ll}
\tau(x,v)/2,\;\;\;\;\;\;\;\;\;(x,v)\in \partial_{+}(SM)\\
-\tau(x,-v)/2,\;\;\;(x,v)\in\partial_{-}(SM)\\
\end{array}\right.\]
is smooth.
Recall that the geodesic flow $\varphi_{t}$ is defined as follows: given $(x,v)\in SM$ we let
$t\mapsto \gamma(t,x,v)$ be the geodesic starting from $x$ in the direction of $v$; then $\varphi_{t}(x,v)=(\gamma(t,x,v),\dot{\gamma}(t,x,v))$.
We also define the scattering relation $\alpha:\partial(SM)\to\partial(SM)$ by 
\[\alpha(x,v)=\varphi_{2\tau_{-}(x,v)}(x,v).\]
Then $\alpha$ is a diffeomorphism and $\alpha^2=\id$.

Let $X$ denote the vector field associated with the geodesic flow.
Since $M$ is assumed oriented there is a circle action on the fibers of $SM$ with infinitesimal generator $V$ called the {\it vertical vector field}. It is possible to complete the pair $X,V$ to a global frame
of $T(SM)$ by considering the vector field $X_{\perp}:=[X,V]$, where $[\cdot,\cdot]$ stands for the Lie bracket or commutator of two vector fields. There are two additional structure equations given by $X=[V,X_{\perp}]$ and $[X,X_{\perp}]=-KV$
where $K$ is the Gaussian curvature of the surface. Using this frame we can define a Riemannian metric on $SM$ by declaring $\{X,X_{\perp},V\}$ to be an orthonormal basis and the volume form of this metric will be denoted by $d\Sigma^3$ and referred to as the {\it usual} volume form on $SM$. The fact that $\{ X, X_{\perp}, V \}$ are orthonormal together with the commutator formulas implies that the Lie derivative of $d\Sigma^3$ along the three vector fields vanishes, and consequently these vector fields are volume preserving.

Given functions $u,v:SM\to \C^n$ we consider the
inner product
\[\langle u,v \rangle =\int_{SM}\langle u,v\rangle_{\C^n}\,d\Sigma^3.\]
The space $L^{2}(SM,\C^n)$ decomposes orthogonally
as a direct sum
\[L^{2}(SM,\C^n)=\bigoplus_{k\in\mathbb Z}H_{k}\]
where $-iV$ acts as $k\,\mbox{\rm Id}$ on $H_k$.
Following Guillemin and Kazhdan in \cite{GK} we introduce the following
first order elliptic operators 
$$\eta_{+},\eta_{-}:C^{\infty}(SM,\C^n)\to
C^{\infty}(SM,\C^n)$$ given by
\[\eta_{+}:=(X+iX_{\perp})/2,\;\;\;\;\;\;\eta_{-}:=(X-iX_{\perp})/2.\]
Clearly $X=\eta_{+}+\eta_{-}$. Let $\Omega_{k}:=C^{\infty}(SM,\C^n)\cap H_{k}$. The commutation relations $[-iV,\eta_{+}]=\eta_{+}$ and
$[-iV,\eta_{-}]=-\eta_{-}$ imply that
\[\eta_{+}:\Omega_{k}\to \Omega_{k+1},\;\;\;\;\eta_{-}:\Omega_{k}\to \Omega_{k-1}.\]
These operators will be particularly useful when a Higgs field is introduced.

A smooth function $u:SM\to\C^n$ has a Fourier series expansion
\[u=\sum_{k=-\infty}^{\infty}u_{k},\]
where $u_{k}\in\Omega_k$.
Locally we can always choose isothermal coordinates $(x_{1},x_{2})$ so that the metric
can be written as $ds^2=e^{2\lambda}(dx_{1}^2+dx_{2}^2)$ where $\lambda$ is a smooth
real-valued function of $x=(x_{1},x_{2})$. This gives coordinates $(x_{1},x_{2},\theta)$ on $SM$ where
$\theta$ is the angle between a unit vector $v$ and $\partial/\partial x_{1}$.
In these coordinates we may write $V=\partial/\partial\theta$ and
$u_{k}=\tilde{u}_{k}(x)e^{ik\theta}$. For completeness we include formulas for $X$ and $X_{\perp}$ even though these
will not be needed in the sequel:
\[X=e^{-\lambda}\left(\cos\theta\frac{\partial}{\partial x_{1}}+
\sin\theta\frac{\partial}{\partial x_{2}}+
\left(-\frac{\partial \lambda}{\partial x_{1}}\sin\theta+\frac{\partial\lambda}{\partial x_{2}}\cos\theta\right)\frac{\partial}{\partial \theta}\right);\]
\[X_{\perp}=-e^{-\lambda}\left(-\sin\theta\frac{\partial}{\partial x_{1}}+
\cos\theta\frac{\partial}{\partial x_{2}}-
\left(\frac{\partial \lambda}{\partial x_{1}}\cos\theta+\frac{\partial \lambda}{\partial x_{2}}\sin\theta\right)\frac{\partial}{\partial \theta}\right).\]

Another important ingredient in our approach is the fibrewise {\it Hilbert transform} $\H$.
This can be introduced in various ways (cf.~\cite{PU,SaU}), but perhaps the most informative approach is to indicate that it acts fibrewise and
 for $u_{k}\in \Omega_k$,
\[\H(u_{k})=-\mbox{\rm sgn}(k)\,iu_{k}\]
where we use the convention $\mbox{\rm sgn}(0)=0$.
Moreover, $\H(u)=\sum_{k}\H(u_{k})$.
Observe that
\[(\id+i\H)u=u_0+2\sum_{k=1}^{\infty}u_{k},\]
\[(\id-i\H)u=u_0+2\sum_{k=-\infty}^{-1}u_{k}.\]

\begin{Definition} A function $u:SM\to\C^n$ is said to be holomorphic if
$(\id-i\H)u=u_{0}$. Equivalently, $u$ is holomorphic if $u_{k}=0$ for all $k<0$. Similarly, $u$ is said to be antiholomorphic if $(\id+i\H)u=u_{0}$ which
is equivalent to saying that $u_{k}=0$ for all $k>0$.
\end{Definition}

As in previous works (cf. \cite{PU,SaU}) the following commutator formula of Pestov-Uhlmann \cite{PU} will play a key role:
\begin{equation}
[\H,X] u = X_{\perp} u_0 + (X_{\perp} u)_0, \quad u \in C^{\infty}(SM,\C^n).
\label{eq:puf}
\end{equation}

We will extend this bracket relation so that it includes a connection $A$ and
a Higgs field $\Phi$.
We often think of $A$ as a function restricted
to $SM$. We also think of $A$ as acting on smooth functions
$u\in C^{\infty}(SM,\C^n)$ by multiplication. Note that $V(A)$ is a new
function on $SM$ which can be identified with the restriction of
$-\star A$ to $SM$, so we will simply write $V(A)=-\star A$. Here $\star$ denotes the Hodge star operator of the metric $g$. We now show:

\begin{Lemma} For any smooth function $u$ we have
\[[\H,X+A+\Phi]u=(\X+\star A)(u_{0})+\{(\X+\star A)(u)\}_{0}.\]
\label{lemma:bracket}
\end{Lemma}

\begin{proof} In view of (\ref{eq:puf}) and $[\H,\Phi]=0$ we only need to prove that
\begin{equation}
[\H,A]u=\star A u_{0}+(\star A u)_{0}.
\label{eq:Ab}
\end{equation}
Expanding $u$ in Fourier series we see that we only need to check the formula for functions $u_k\in\Omega_k$.
In this case
\[\H(u_k)=-\mbox{\rm sgn}(k)iu_k,\]
where as above we use the convention $\mbox{\rm sgn}(0)=0$.
Since $A$ is a 1-form its Fourier expansion is given by:
\[A=A_{-1}+A_{1}.\]
We compute:
\begin{align*}
[\H,A]u_k&=\H(A_{-1}u_{k}+A_{1}u_k)-A(-\mbox{\rm sgn}(k)iu_k)\\
&=-\mbox{\rm sgn}(k-1)iA_{-1}u_{k}-\mbox{\rm sgn}(k+1)iA_{1}u_{k}+A(\mbox{\rm sgn}(k)iu_{k})\\
&=iA_{-1}u_{k}(\mbox{\rm sgn}(k)-\mbox{\rm sgn}(k-1))+iA_{1}u_{k}(\mbox{\rm sgn}(k)-\mbox{\rm sgn}(k+1)).
\end{align*}
Let us analyze this expression for various values of $k$. If $|k|\geq 2$
we see that $[\H,A]u_k=0$ which proves (\ref{eq:Ab}) since
its right hand side clearly vanishes.
If $k=0$, $[\H,A]u_0=iA_{-1}u_0-iA_{1}u_0$.
On the other hand
\[\star A=-V(A)=iA_{-1}-iA_{1}.\]
Thus
\[[\H,A]u_{0}=\star A u_0+(\star A u_0)_{0}=\star A u_0\]
which proves (\ref{eq:Ab}) in this case.
If $k=1$, then $[\H,A]u_{1}=iA_{-1}u_1$.
But
\[(\star A u_{1})_{0}=iA_{-1}u_{1}\]
and $(u_{1})_{0}=0$ which also shows (\ref{eq:Ab}). The case $k=-1$
is checked in the same way.
\end{proof}

\section{The attenuated ray transform and scattering data}
\label{scatt}

We wish to give a precise definition of the attenuated geodesic ray transform and scattering data in the presence of a connection and a Higgs field. In this section we assume that $(M,g)$ is a nontrapping compact manifold, of dimension $d \geq 2$, with strictly convex boundary.

First recall that in the scalar case, the attenuated ray transform $I_a f$ of a function $f \in C^{\infty}(SM,\C)$ with attenuation coefficient $a \in C^{\infty}(SM,\C)$ can be defined as the integral 
$$
I_a f(x,v) := \int_0^{\tau(x,v)} f(\varphi_t(x,v)) \text{exp}\left[ \int_0^t a(\varphi_s(x,v)) \,ds \right] dt, \quad (x,v) \in \partial_+(SM).
$$
Alternatively, we may set $I_a f := u|_{\partial_+ (SM)}$ where $u$ is the unique solution of the transport equation 
$$
Xu + au = -f \ \ \text{in $SM$}, \quad u|_{\partial_-(SM)} = 0.
$$

The last definition generalizes without difficulty to the case of systems. Assume that  $\mathcal{A} \in C^{\infty}(SM,\C^{n \times n})$ is a matrix attenuation coefficient and let $f \in C^{\infty}(SM,\C^n)$ be a vector valued function. At this point both $\mathcal{A}$ and $f$ may be quite arbitrary, although in this paper we will mostly be interested in the special case 
$$
\mathcal{A}(x,v) = \Phi(x) + A_j(x) v^j, \quad f(x,v) = F(x) + \alpha_j(x) v^j.
$$
Here $A_j \,dx^j$ is a $\C^{n \times n}$-valued smooth $1$-form (called the \emph{connection}) and $\Phi$ is a $\C^{n \times n}$ valued smooth function on $M$ (called the \emph{Higgs field}), and $F \in C^{\infty}(M,\C^n)$ and $\alpha = \alpha_j \,dx^j$ is a  $\C^n$-valued $1$-form. Consider the following transport equation for $u: SM \to \C^n$, 
$$
Xu + \mathcal{A}u = -f \ \ \text{in $SM$}, \quad u|_{\partial_-(SM)} = 0.
$$
On a fixed geodesic the transport equation becomes a linear system of ODEs with zero initial condition, and therefore this equation has a unique solution $u=u^f$.

\begin{Definition}
The geodesic ray transform of $f \in C^{\infty}(SM,\C^n)$ with attenuation $\mathcal{A} \in C^{\infty}(SM,\C^{n \times n})$ is given by 
$$
I_{\mathcal{A}} f := u^f|_{\partial_+(SM)}.
$$
In the special case $\mathcal{A}(x,v) = A_j(x) v^j + \Phi(x)$ we write 
$$
I_{A,\Phi} f := I_{\mathcal{A}} f.
$$
\end{Definition}

We note that $I_{A,\Phi}$ acting on sums of $0$-forms and $1$-forms always has a nontrivial kernel, since 
$$
I_{A,\Phi}((X+A+\Phi)p) = 0 \text{ for any $p \in C^{\infty}(M,\C^n)$ with $p|_{\partial M} = 0$}.
$$
Thus from the ray transform $I_{A,\Phi} f$ one only expects to recover $f$ up to an element having this form.

The transform $I_{\mathcal{A}}$ also has an integral representation. Consider the unique matrix solution
$U_{-}:SM\to GL(n,\C)$ to the transport equation:
\[\left\{\begin{array}{ll}
XU_{-}+\mathcal{A}U_{-}=0,\\
U_{-}|_{\partial_{+}(SM)}=\id.\\
\end{array}\right.\]
Similarly, we can consider the unique matrix solution of 
\[\left\{\begin{array}{ll}
XU_{+}+\mathcal{A}U_{+}=0,\\
U_{+}|_{\partial_{-}(SM)}=\id.\\
\end{array}\right.\]
Note that $U_{-}^{-1}$ solves $X(U_{-}^{-1})-U_{-}^{-1}\mathcal{A}=0$. This implies that 
\[X(U_{-}^{-1}u^f)=X(U_{-}^{-1})u^f+U_{-}^{-1}X(u^f)
=U_{-}^{-1}(\mathcal{A}u^f-f-\mathcal{A}u^f)=-U_{-}^{-1}f.\]
Therefore integrating between $0$ and $\tau(x,v)$ for $(x,v)\in \partial_{+}(SM)$ we derive the integral representation 
$$
I_{\mathcal{A}} f(x,v) = \int_{0}^{\tau(x,v)} U_{-}^{-1}(\varphi_{t}(x,v))f(\varphi_{t}(x,v))\,dt.
$$

Recall that in the nontrapping manifold $(M,g)$ we have the scattering relation $\alpha:\partial (SM)\to \partial (SM)$, mapping an initial point and direction $(x,v)$ of a geodesic to the end point and direction. Given a matrix attenuation coefficient $\mathcal{A} \in C^{\infty}(SM,\C^{n \times n})$ we also have a corresponding parallel transport type operator, which takes a vector $h(0) \in \C^n$ to the vector $h(\tau(x,v)) \in \C^n$ where 
$$
h'(t) + \mathcal{A}(\varphi_t(x,v)) h(t) = 0 \ \ \text{for $t \in [0,\tau(x,v)]$}.
$$
It is easy to see that $h(\tau(x,v))$ is the same as $U_-(\alpha(x,v)) h(0)$, and so the operator $h(0) \mapsto h(\tau(x,v))$ for $(x,v) \in \partial_+(SM)$ is completely determined by $U_-|_{\partial_-(SM)}$ and $\alpha$.

\begin{Definition}
The scattering data corresponding to a matrix attenuation coefficient $\mathcal{A}$ in $(M,g)$ is the map 
$$
C^{\mathcal{A}}_{-}:\partial_{-}(SM)\to GL(n,\C), \ \  C^{\mathcal{A}}_{-} := U_{-}|_{\partial_{-}(SM)}.
$$
In the special case $\mathcal{A}(x,v) = A_j(x) v^j + \Phi(x)$ we write 
$$
C^{A,\Phi}_{-} := C^{\mathcal{A}}_{-}.
$$
\end{Definition}

There is also a corresponding map $C^{\mathcal{A}}_{+}:\partial_{+}(SM)\to GL(n,\C), \  C^{\mathcal{A}}_{+}:=U_{+}|_{\partial_{+}(SM)}$. However, one has 
\[X(U_{-}^{-1}U_{+})=X(U_{-}^{-1})U_{+}+U_{-}^{-1}X(U_{+})=U_{-}^{-1} \mathcal{A} U_{+}-U_{-}^{-1} \mathcal{A} U_{+}=0.\]
Thus
\[U_{-}^{-1}U_{+}|_{\partial_{-}(SM)}=(U_{-}^{-1}U_{+})\circ\alpha |_{\partial_{-}(SM)}\]
which gives
$$
(C^{\mathcal{A}}_{-})^{-1}=C^{\mathcal{A}}_{+}\circ\alpha.
$$
Consequently $C^{\mathcal{A}}_{+}$ is uniquely determined by $C^{\mathcal{A}}_{-}$ and $\alpha$.

Finally we discuss gauge invariance of the transport equation in the case where $\mathcal{A}(x,v) = A_j(x)v^j + \Phi(x)$. Suppose $u = u^f$ solves 
$$
(X+A+\Phi)u = -f \ \ \text{in $SM$}, \quad u|_{\partial_-(SM)} = 0.
$$
If $Q: M \to GL(n,\C)$ is a smooth function taking values in invertible matrices, then $Q^{-1} u$ satisfies 
$$
X(Q^{-1} u) = Q^{-1}(Xu - (XQ)Q^{-1} u) = -Q^{-1}(Au + \Phi u + f + (XQ)Q^{-1} u)
$$
and consequently 
$$
(X + Q^{-1}AQ + Q^{-1}XQ + Q^{-1} \Phi Q)(Q^{-1} u) = -Q^{-1} f.
$$
Therefore $Q^{-1} u$ satisfies a transport equation with new connection $Q^{-1}(X+A)Q$ and new Higgs field $Q^{-1} \Phi Q$. It also follows that 
$$
I_{Q^{-1}(X+A)Q,Q^{-1} \Phi Q}(Q^{-1} f) = Q^{-1}|_{\partial_+(SM)} I_{A,\Phi} f.
$$
Thus, when proving injectivity results for $I_{A,\Phi}$ we have the freedom to change the connection and Higgs field into gauge equivalent ones. This idea will be useful in our proof of injectivity on compact simple surfaces, where the connection will be gauge transformed to one whose curvature has a definite sign (however the gauge transform will be more general and $Q$ will depend on direction as well).

As for the scattering data, if $U_-$ solves $(X+A+\Phi)U_- = 0$ in $SM$ with $U_{-}|_{\partial_+(SM)} = \id$, then $Q^{-1}U_-$ satisfies 
$$
(X+Q^{-1}(X+A)Q + Q^{-1} \Phi Q)(Q^{-1}U_-) = 0 \ \ \text{in $SM$}, \ \ \ Q^{-1} U_-|_{\partial_+(SM)} = Q^{-1}|_{\partial_+(SM)}.
$$
The scattering data has therefore the gauge invariance 
$$
C^{Q^{-1}(X+A)Q,Q^{-1} \Phi Q}_- = C^{A,\Phi}_- \quad \text{if $Q \in C^{\infty}(M,GL(n,\C))$ satisfies $Q|_{\partial M} = \id$}.
$$
It follows that from the knowledge of $C^{A,\Phi}_{-}$ one can only expect to recover $A$ and $\Phi$ up to a gauge transform via $Q$ which satisfies $Q|_{\partial M} = \id$.

\section{Holomorphic integrating factors in the scalar case}
\label{sec:hol}

Let $(M,g)$ be a 2D simple manifold. In this section we will prove the existence of holomorphic and antiholomorphic integrating factors for the equation 
$$
Xu + \mathcal{A} u = -f \quad \text{ in $SM$}
$$
in the scalar case where $n=1$. We assume that $\mathcal{A}(x,v) = \Phi(x) + \alpha_j(x)v^j$ where $\Phi$ is a smooth complex valued function on $M$ and $\alpha$ is a complex $1$-form. The discussion is parallel to \cite{SaU} which considered the case $\alpha = 0$. All functions in this section will be scalar and complex valued.

By a holomorphic (resp.~antiholomorphic) integrating factor we mean a function $e^{-w}$, where $w \in C^{\infty}(SM)$ is holomorphic (resp.~antiholomorphic) in the angular variable, such that for all $r\in C^{\infty}(SM)$
$$
e^w X(e^{-w} r) = Xr + \mathcal{A} r \quad \text{in $SM$}.
$$
This is equivalent with the equation 
$$
Xw = -\mathcal{A} \quad \text{in $SM$}.
$$
The main result of this section shows that when $\mathcal{A}$ is the sum of a function and $1$-form, holomorphic and antiholomorphic integrating factors always exist.

\begin{Theorem} \label{prop_holomorphic_integrating_factors}
Let $(M,g)$ be a simple two-dimensional manifold and  $f \in C^{\infty}(SM)$. The following conditions are equivalent.
\begin{enumerate}
\item[(a)]
There exist a holomorphic $w \in C^{\infty}(SM)$ and antiholomorphic $\tilde{w} \in C^{\infty}(SM)$ such that $Xw = X\tilde{w} = -f$.
\item[(b)]
$f(x,v) = F(x) + \alpha_j(x) v^j$ where $F$ is a smooth function on $M$ and $\alpha$ is a $1$-form.
\end{enumerate}
Furthermore, the following conditions are equivalent.
\begin{enumerate}
\item[($\alpha$)]
There is a holomorphic (resp.~antiholomorphic) $w \in C^{\infty}(SM)$ with $Xw = -f$.
\item[($\beta$)]
$f = f' + f''$ where $f' \in C^{\infty}(SM)$ is holomorphic (resp.~antiholomorphic) and $f''(x,v) = \alpha_j(x)v^j$ for some $1$-form $\alpha$.
\end{enumerate}
\end{Theorem}

We begin with some regularity results which are valid on any nontrapping compact manifold $(M,g)$ with strictly convex boundary. If $f \in C^{\infty}(SM)$ define 
$$
u^f(x,v) = \int_0^{\tau(x,v)} f(\varphi_t(x,v)) \,dt, \quad (x,v) \in SM,
$$
where $\tau$ is the exit time of the geodesic determined by $(x,v)$. Note that $\tau$ is continuous on $SM$ and smooth in $SM \setminus S(\partial M)$, but may fail to be smooth in $SM$. The same properties are true for $u^f$. We also have that $\tau_-$ is smooth in $SM$. Recall the definition 
$$
C^{\infty}_{\alpha}(\partial_+ (SM)) = \{ h \in C^{\infty}(\partial_+(SM)) \,;\, h_{\psi} \in C^{\infty}(SM) \}
$$
where 
$$
h_{\psi}(x,v) = h(\varphi_{-\tau(x,-v)}(x,v)).
$$
If $u$ is a function on $SM$ we denote the even and odd parts with respect to $v$ by 
$$
u_+(x,v) = \frac{1}{2}(u(x,v) + u(x,-v)), \quad u_-(x,v) = \frac{1}{2}(u(x,v) - u(x,-v)).
$$

\begin{Lemma} \label{lemma_parity_smoothness}
Let $(M,g)$ be a nontrapping compact manifold with strictly convex boundary and let $b \in C^{\infty}(SM)$. If $b$ is even then $u^b_-$ is smooth in $SM$. Similarly, if $b$ is odd then $u^b_+$ is smooth in $SM$.
\end{Lemma}
\begin{proof}
The case where $b$ is even was proved in \cite[Lemma 2.3]{SaU}, and the odd case follows by repeating that proof almost verbatim. For completeness we give the details. The proof is based on \cite[Lemma 1.1]{PU} which gives the characterization 
$$
C^{\infty}_{\alpha}(\partial_+(SM)) = \{ h \in C^{\infty}(\partial_+(SM)) \,;\, E h \in C^{\infty}(\partial(SM)) \}
$$
where $E$ is the operator of even continuation with respect to the scattering relation, 
$$
E h(x,v) = \left\{ \begin{array}{cl} h(x,v), & (x,v) \in \partial_+(SM), \\ h(\alpha(x,v)), & (x,v) \in \partial_-(SM). \end{array} \right.
$$

Let $(\tilde{M},g)$ be a nontrapping manifold with strictly convex boundary chosen so that $M \subseteq \tilde{M}^{\text{int}}$ (this can be achieved by embedding $(M,g)$ to a compact manifold $(S,g)$ without boundary and by looking at a small neighborhood of $M$ in $S$). If $\tilde{\tau}(x,v)$ is the exit time of geodesics in $(\tilde{M},g)$, we know that $\tilde{\tau}$ is smooth in $S(\tilde{M}^{\text{int}})$.

Extend $b$ as a smooth odd function into $S \tilde{M}$, and define 
\begin{equation*}
\tilde{u}(x,v) = \int_0^{\tilde{\tau}(x,v)} b(\varphi_t(x,v)) \,dt
\end{equation*}
where $\varphi_t$ is the geodesic flow in $(\tilde{M},g)$. Then $\tilde{u} \in C^{\infty}(SM)$ and $X \tilde{u} = -b$ in $SM$.

Let $h = (\tilde{u} - u^b_+)|_{\partial_+ (SM)}$. Since $X (\tilde{u} - u^b_+) = 0$ in $S(M^{\text{int}})$ (here we used that $b$ is odd) and $\tilde{u} - u^b_+$ is continuous in $SM$, we obtain $\tilde{u} - u^b_+ = h_{\psi}$. Thus, to show that $u^b_+$ is smooth in $SM$ it is enough to prove that $E h$ is in $C^{\infty}(\partial (SM))$. We have for $(x,v) \in \partial (SM)$ 
\begin{equation*}
E h(x,v) = \frac{1}{2} \int_0^{\tilde{\tau}(x,v)} b(\varphi_t(x,v)) \,dt + \frac{1}{2} \int_{2\tau_-(x,v)}^{\tilde{\tau}(x,v)} b(\varphi_t(x,v)) \,dt.
\end{equation*}
Here $2\tau_-$ is smooth in $\partial (SM)$ and so also $E h$ is smooth.
\end{proof}

The next step is an abstract characterization of those functions $f$ on $SM$ for which the equation $Xw = -f$ admits a holomorphic solution $w$.

\begin{Lemma} \label{lemma_abstract_characterization}
Let $(M,g)$ be a nontrapping compact surface with strictly convex boundary, and let $f \in C^{\infty}(SM)$. The following are equivalent.
\begin{enumerate}
\item[(1)]
There exists a holomorphic function $w \in C^{\infty}(SM)$ with $Xw = -f$.
\item[(2)]
There exist $b \in C^{\infty}(SM)$ and $h, h' \in C^{\infty}_{\alpha}(\partial_+(SM))$ such that 
$$
f = (\id + i\hilb)b + i X_{\perp} (u^{b_-})_0 + i X_{\perp} (h_{\psi})_0 + i(X_{\perp} u^{b_+})_0 + i(X_{\perp} h'_{\psi})_0.
$$
\end{enumerate}
\end{Lemma}
\begin{proof}
Suppose $w$ is holomorphic and $Xw = -f$. Write $w = (\id + i\hilb) \hat{w}$ where $\hat{w} = \frac{1}{2}(w + w_0)$, so $\hat{w} \in C^{\infty}(SM)$. Then, using the commutator formula for $\hilb$ and $X$, 
\begin{align*}
-f &= X(\id + i\hilb) \hat{w} = (\id + i\hilb) X\hat{w} - i[\hilb,X] \hat{w} \\
 &= (\id + i\hilb) X\hat{w} - i X_{\perp} \hat{w}_0 - i(X_{\perp} \hat{w})_0.
\end{align*}
Write $b = -X \hat{w}$. Separating the even and odd parts we have $X \hat{w}_+ = -b_-$ and $X \hat{w}_- = -b_+$, and consequently 
$$
\hat{w}_+ = u^{b_-}_+ + (h_\psi)_+, \quad \hat{w}_- = u^{b_+}_- + (h_{\psi}')_-
$$
for some $h, h' \in C^{\infty}_{\alpha}(\partial_+(SM))$. (Here we use that $u^{b_+}_-$ and $u^{b_-}_+$ are smooth in $SM$ by Lemma \ref{lemma_parity_smoothness}.) It follows that $f$ can be written in the required form.

Conversely, assume that $f$ has the stated form. Let 
$$
\hat{w} = u^{b_-}_+ + (h_\psi)_+ + u^{b_+}_- + (h_{\psi}')_-.
$$
Then $\hat{w} \in C^{\infty}(SM)$ and $X\hat{w} = -b$. Define $w = (\id + i\hilb)\hat{w}$, so that $w \in C^{\infty}(SM)$ is holomorphic. Then 
\begin{align*}
Xw = (\id + i\hilb) X\hat{w} - i[\hilb,X] \hat{w} = -(\id+i\hilb)b - iX_{\perp} \hat{w}_0 - i(X_{\perp} \hat{w})_0.
\end{align*}
This shows that $Xw = -f$.
\end{proof}

The characterization of those $f$ which admit antiholomorphic solutions is analogous.

\begin{Lemma} \label{lemma_abstract_characterization2}
Let $(M,g)$ be a nontrapping compact surface with strictly convex boundary, and let $f \in C^{\infty}(SM)$. The following are equivalent.
\begin{enumerate}
\item[(1)]
There exists an antiholomorphic $\tilde{w} \in C^{\infty}(SM)$ with $X\tilde{w} = -f$.
\item[(2)]
There exist $b \in C^{\infty}(SM)$ and $h, h' \in C^{\infty}_{\alpha}(\partial_+(SM))$ such that 
$$
f = (\id - i\hilb)b - i X_{\perp} (u^{b_-})_0 - i X_{\perp} (h_{\psi})_0 - i(X_{\perp} u^{b_+})_0 - i(X_{\perp} h'_{\psi})_0.
$$
\end{enumerate}
\end{Lemma}

We will focus on the following maps appearing in Lemmas \ref{lemma_abstract_characterization} and \ref{lemma_abstract_characterization2},
\begin{gather}
C^{\infty}_{\alpha}(\partial_+(SM)) \to C^{\infty}(M), \ \ h \mapsto (X_{\perp} h_{\psi})_0, \label{surjective_map1} \\
C^{\infty}_{\alpha}(\partial_+(SM)) \to C^{\infty}_{\delta}(M,\Lambda^1), \ \ h \mapsto X_{\perp}(h_{\psi})_0. \label{surjective_map2}
\end{gather}
Here $C^{\infty}_{\delta}(M,\Lambda^1)$ is the space of solenoidal complex $1$-forms on $M$ (that is, $1$-forms $\alpha$ satisfying $\delta \alpha = 0$, where $\delta$ is the codifferential). Note that if $F \in C^{\infty}(M)$ then $X_{\perp} F(x,v) = (\star dF)_j(x)v^j$ where $\star dF$ is indeed a solenoidal $1$-form.

It was proved in \cite{SaU} that the map \eqref{surjective_map1} is surjective when $(M,g)$ is a simple two-dimensional manifold. We need an analogous result for \eqref{surjective_map2}.

\begin{Proposition} \label{prop_surjectivity}
If $(M,g)$ is a simple surface, the map \eqref{surjective_map2} is surjective.
\end{Proposition}

Assuming this, we can present the proof of the main result of this section.

\begin{proof}[Proof of Theorem \ref{prop_holomorphic_integrating_factors}]
We first prove that ($\alpha$) implies ($\beta$). If $Xw = -f$ where $w$ is (anti)holomorphic, then Lemmas \ref{lemma_abstract_characterization} and \ref{lemma_abstract_characterization2} show that $f$ is the sum of an (anti)holomorphic function and a $1$-form as required.

Conversely, let $f = f' + f''$ where $f'$ is holomorphic and $f''$ is a $1$-form. We assume for the moment that $f''$ is solenoidal. Following Lemma \ref{lemma_abstract_characterization} we take $b = \frac{1}{2}(f' + (f')_0)$, so that $(\id + i\H) b = f'$. It is then enough to find $h, h' \in C^{\infty}_{\alpha}(\partial_+(SM))$ such that 
$$
iX_{\perp}(h_{\psi})_0 + i(X_{\perp} h_{\psi}')_0 = f'' - i X_{\perp}(u^{b_-})_0 - i(X_{\perp} u^{b_+})_0.
$$
The right hand side is the sum of a function on $M$ and a solenoidal $1$-form, and the surjectivity of the maps \eqref{surjective_map1} and \eqref{surjective_map2} imply that there exist $h$ and $h'$ which have the required property. Lemma \ref{lemma_abstract_characterization} shows the existence of a holomorphic solution $w$.

If $f''(x,v) = \alpha_j(x)v^j$ is a general $1$-form we decompose $\alpha = \alpha^s + dp$ where $\alpha^s$ is a solenoidal $1$-form and $p$ is a smooth function on $M$, and note that $Xw = -f$ is equivalent with $Xw' = -f' - (\alpha^s)_j v^j$ for $w' = w+p$. We have already proved that the last equation has a holomorphic solution $w'$, and then $w'-p$ is the required solution of the original equation. This proves that ($\beta$) implies ($\alpha$) in the holomorphic case, and the antiholomorphic case is proved similarly by Lemma \ref{lemma_abstract_characterization2}.

It follows immediately that (b) implies (a). Also, if (a) is valid then $f = f' + f'' = \tilde{f}' + \tilde{f}''$ where $f'$ is holomorphic, $\tilde{f}'$ is antiholomorphic, and $f''$ and $\tilde{f}''$ are $1$-forms. By looking at Fourier coefficients one obtains that $f$ is the sum of a $0$-form and a $1$-form.
\end{proof}

It remains to prove Proposition \ref{prop_surjectivity}. The map \eqref{surjective_map2} may be identified with the operator 
\begin{equation} \label{s_operator_definition}
S: C^{\infty}_{\alpha}(\partial_+ (SM)) \to C^{\infty}_{\delta}(M,\Lambda^1), \ \ Sh = \star d(h_{\psi})_0.
\end{equation}
It is well known that $(h_{\psi})_0 = I^* h$ where $I: L^2(M) \to L^2(\partial_+(SM))$ is the usual geodesic ray transform acting on functions, and $L^2(\partial_+(SM))$ is equipped with the volume form $-\langle v, \nu(x) \rangle \,d(\partial (SM))$ (see \cite{PU}).

\begin{proof}[Proof of Proposition \ref{prop_surjectivity}]
Let $\beta \in C^{\infty}_{\delta}(M,\Lambda^1)$ be a solenoidal $1$-form. Since $d\star\beta = 0$ and $M$ is simply connected, there is $F \in C^{\infty}(M)$ such that $dF = -\star\beta$. From \cite{PU} we know that the map 
$$
I^*: C^{\infty}_{\alpha}(\partial_+(SM)) \to C^{\infty}(M)
$$
is surjective. Take $h \in C^{\infty}_{\alpha}(\partial_+(SM))$ so that $I^* h = F$. It follows that $Sh = \star d(I^* h) = \star dF = \beta$ as required.
\end{proof}

\section{A regularity result for the transport equation} \label{sec:regularity}

The purpose of this section is to show a regularity result for the
transport equation which will be useful later on. The result is
 quite general (requires $M$ non-trapping and with strictly convex boundary)
 and essentially says that if a smooth function is in the
kernel of the attenuated ray transform, the solution to the associated transport equation must be smooth in all $SM$.

Given a smooth $w\in C^{\infty}(\partial_{+}(SM),\C^n)$ consider
the unique solution $w^{\sharp}:SM\to \C^n$
to the transport equation:
\[\left\{\begin{array}{ll}
X(w^{\sharp})+{\mathcal A}w^{\sharp}=0,\\
w^{\sharp}|_{\partial_{+}(SM)}=w.\\
\end{array}\right.\] 
Observe that
\[w^{\sharp}(x,v)=U_{-}(x,v)w(\alpha\circ\psi(x,v))\,\]
where $\psi(x,v):=\varphi_{\tau(x,v)}(x,v)$ (recall that $\varphi_{t}$ is the geodesic flow and $\tau(x,v)$ is the time it takes the geodesic determined by $(x,v)$ to exit $M$).
If we introduce an operator 
\[Q:C(\partial_{+}(SM),\C^n)\to C(\partial(SM),\C^n)\]
by setting
\[Qw(x,v)=\left\{\begin{array}{ll}
w(x,v)&\mbox{\rm if}\;(x,v)\in\partial_{+}(SM)\\
C^{\mathcal A}_{-}(x,v)(w\circ\alpha)(x,v)&\mbox{\rm if}\;(x,v)\in\partial_{-}(SM),\\
\end{array}\right.\] 
then
\[w^{\sharp}|_{\partial(SM)}=Qw.\]
Define
\[\mathcal S^{\infty}(\partial_{+}(SM),\C^n):=\{w\in C^{\infty}(\partial_{+}(SM),\C^n):\;w^{\sharp}\in C^{\infty}(SM,\C^n)\}.\]

We characterize this space purely in terms of scattering data as follows:

\begin{Lemma} The set of those smooth $w$ such that $w^{\sharp}$ is smooth
is given by 
\[\mathcal S^{\infty}(\partial_{+}(SM),\C^n)=\{w\in C^{\infty}(\partial_{+}(SM),\C^n):\;Qw\in C^{\infty}(\partial(SM),\C^n)\}.\]
\label{lemma:ss}
\end{Lemma}
\begin{proof} We will reduce this lemma to Lemma 1.1 in \cite{PU}. Let us recall its statement.
The solution of the boundary value problem for the transport equation:
\[X(u)=0,\;\;\;u|_{\partial_{+}(SM)}=s\]
can be written in the form
\[u=s_{\psi}:=s\circ\alpha\circ\psi.\]
As in Section \ref{sec:hol}, define the extension operator
\[E:C(\partial_{+}(SM),\C^n)\to C(\partial(SM),\C^n)\]
by setting
\[Es(x,v)=\left\{\begin{array}{ll}
s(x,v)&\mbox{\rm if}\;(x,v)\in\partial_{+}(SM)\\
s\circ\alpha(x,v)&\mbox{\rm if}\;(x,v)\in\partial_{-}(SM),\\
\end{array}\right.\] 
then
\[s_{\psi}|_{\partial(SM)}=Es.\]
Lemma 1.1 in \cite{PU} asserts that $s_{\psi}$ is smooth if and only
if $Es$ is smooth.

Embed $M$ into a closed manifold $S$ and extend smoothly
the metric $g$ to $S$ and the attenuation $\mathcal A$ to the unit circle 
bundle of $S$. Let $W$ be an open neighborhood of $M$ in $S$. If $W$ is small enough, then $\overline{W}$ will be non-trapping and with strictly convex boundary
and assume such a $W$ is chosen.
Consider the unique solution to the transport equation in $S\overline{W}$:
\[\left\{\begin{array}{ll}
X(R)+\mathcal{A}R=0,\\
R|_{\partial_{+}(S\overline{W})}=\id.\\
\end{array}\right.\] 
If we restrict $R$ to $SM$ we obtain a {\it smooth} map
and we denote this restriction also by $R$. In fact for what follows any smooth
solution $R:SM\to GL(n,\C)$ to $X(R)+\mathcal{A}R=0$ will do.
Define $p:=R^{-1}|_{\partial_{+}(SM)}$. The main observation is that we can write
\[w^{\sharp}=R(pw)_{\psi}.\]
Also we have the following expression for $Q$:
\[Qw(x,v)=\left\{\begin{array}{ll}
R(x,v)p(x,v)w(x,v)&\mbox{\rm if}\;(x,v)\in\partial_{+}(SM)\\
R(x,v)((pw)\circ\alpha)(x,v)&\mbox{\rm if}\;(x,v)\in\partial_{-}(SM),\\
\end{array}\right.\] 
Obviously, if $w^{\sharp}$ is smooth, then $Qw=w^{\sharp}|_{\partial(SM)}$ is also smooth.
Assume now that $Qw$ is smooth.
Since $R$ is smooth $R^{-1}Qw=Epw$ is also smooth. As explained above
\cite[Lemma 1.1]{PU} asserts that $(pw)_{\psi}$ is smooth.
Once again, since $R$ is smooth it follows that $w^{\sharp}$ is smooth.
\end{proof}

Given a smooth
$f:SM\to\C^n$, recall that $u^f$ is the unique solution to
\[\left\{\begin{array}{ll}
X(u)+\mathcal{A}u=-f,\\
u|_{\partial_{-}(SM)}=0.\\
\end{array}\right.\]
The function $u^f$ may fail to be differentiable at $S(\partial M)$
due to the non-smoothness of $\tau$. However, we now show that
if  $f$ is in the kernel of $I_{\mathcal{A}}$ then $u^f$ is smooth on $SM$.

\begin{Proposition} Let $f:SM\to\C^n$ be smooth with $I_{\mathcal{A}}(f)=0$.
Then $u^f:SM\to \C^n$ is smooth.
\label{prop:smooth}
\end{Proposition}

\begin{proof} As in the proof of Lemma \ref{lemma:ss} embed $M$ into a closed manifold $S$ and extend smoothly
the metric $g$ to $S$, and $f$ and $\mathcal{A}$ to the unit circle bundle of $S$. Let $W$ be an open neighborhood of $M$ in $S$. If $W$ is small enough, then $\overline{W}$ will be non-trapping and with strictly convex boundary
and assume such a $W$ is chosen.
Consider the unique solution to the following problem in $S\overline{W}$:
\[\left\{\begin{array}{ll}
X(r)+\mathcal{A}r=-f,\\
r|_{\partial_{-}(S\overline{W})}=0.\\
\end{array}\right.\]
Then, the restriction of $r$ to $SM$, still denoted by $r$, 
is a smooth solution $r:SM\to\C^n$ to $(X+\mathcal{A})(r)=-f$.
Observe that $r-u^f$ solves $(X+\mathcal{A})(r-u^f)=0$, thus if we let
$w:=(r-u^f)|_{\partial_{+}(SM)}$, then $w^{\sharp}=r-u^f$.
It follows that $u^f$ is smooth if and only if $w^{\sharp}$ is.
But smoothness of $w^{\sharp}$ is characterized by Lemma \ref{lemma:ss}.
If $u^f|_{\partial_{+}(SM)}=I_{\mathcal{A}}(f)=0$, then $Qw=w^{\sharp}|_{\partial(SM)}=r|_{\partial(SM)}$ which is smooth. It follows that $u^f$ is smooth.
\end{proof}

\section{Injectivity properties of $I_{A}$} \label{sec:injective}

Recall that $SM$ has a
canonical framing $\{X,X_{\perp},V\}$, where $X$ is the geodesic vector field, $V$
is the vertical vector field and $X_{\perp}=-[V,X]$. Recall also
that $X=[V,X_{\perp}]$ and $[X,X_{\perp}]=-KV$, where $K$ is the Gaussian curvature of the surface.

Let $d\Sigma^3$ be the usual volume form in $SM$. Given functions $u,v:SM\to \C^n$ we consider the
inner product
\[\langle u,v \rangle =\int_{SM}\langle u,v\rangle_{\C^n}\,d\Sigma^3.\]

Assume $A$ is a unitary (or Hermitian) connection, in other words, $A$ takes values in the set $\mathfrak u(n)$ of skew-Hermitian matrices.
Recall that the curvature $F_{A}$ of the connection $A$ is defined
as $F_{A}=dA+A\wedge A$. Then $\star F_{A}:M\to \mathfrak u(n)$.
Let $v\in T_{x}M$ be a unit vector and let $iv\in T_{x}M$ be the unique
unit vector such that $\{v,iv\}$ is an oriented orthonormal basis of $T_{x}M$. Then
$\star F_{A}(x)=F_{A}(v,iv)=dA(v,iv)+A(v)A(iv)-A(iv)A(v)$. On the other hand
it is easy to check that $A(v)A(iv)-A(iv)A(v)=[\star A,A](v)$ and 
\[dA(v,iv)=(X_{\perp}(A)-X(\star A))(x,v)\]
hence
\begin{equation}
\star F_{A}=X_{\perp}(A)-X(\star A)+[\star A,A].
\label{eq:curv}
\end{equation}

We will also have a Higgs field $\Phi:M\to\mathfrak u(n)$.
Recall that the connection $A$ induces a covariant derivative $d_{A}$ on endomorphisms so that $d_{A}\Phi=d\Phi+[A,\Phi]$.

We will begin by establishing an energy identity or a ``Pestov type identity'', which generalizes the standard Pestov identity \cite{Sh} to the case where a connection and Higgs field are present. There are several predecessors for
this formula \cite{V,Sha}, but in the form stated below it appears to be new. 
The derivation of the identity and its use for simple surfaces is in the spirit of \cite{SU,DP}.

\begin{Lemma}[Energy identity] If $u:SM\to\C^n$ is a smooth function
such that $u|_{\partial(SM)}=0$, then
\begin{multline*}
|(X+A+\Phi)(Vu)|^2-\langle K\,V(u),V(u)\rangle-\langle \star F_{A}u,V(u)\rangle-\Re\langle (\star d_{A}\Phi) u,V(u)\rangle\\-\Re\langle\Phi u,(X+A+\Phi)u\rangle=|V[(X+A+\Phi)(u))]|^2-|(X+A+\Phi)(u)|^2.
\end{multline*}
\label{lemma:pestov}
\end{Lemma}

Before proving the identity we need a preliminary lemma.

\begin{Lemma} For any pair of smooth functions
$u,g:SM\to\C^n$ we have
\[\langle V(u),g\rangle=-\langle u,V(g)\rangle.\]
Moreover, if in addition $u|_{\partial(SM)}=0$, then
\[\langle Pu,g\rangle=-\langle u,Pg\rangle\]
where $P=X+A+\Phi$ or $P=X_{\perp}+\star A$.
\label{lemma:anti}
\end{Lemma}

\begin{proof} The Lie derivative of the volume form $d\Sigma^3$ along the vector fields $X$, $X_{\perp}$ and $V$ is zero, therefore for $Y=X,X_{\perp},V$ and any smooth function $f$ we have
\[\int_{SM}Y(f)\,d\Sigma^3=\int_{\partial(SM)}f\,i_{Y}(d\Sigma^3).\]
Since $i_{V}(d\Sigma^3)$ vanishes on $\partial(SM)$ the first claim of the lemma follows.

If we assume that $u$ vanishes on $\partial(SM)$ we deduce similarly that
\[\langle X(u),g\rangle=-\langle u,X(g)\rangle,\]
\[\langle X_{\perp}(u),g\rangle=-\langle u,X_{\perp}(g)\rangle.\]
Combining this with the fact that $A+\Phi$ and $\star A$ are skew-Hermitian matrices the second claim of the lemma easily follows.
\end{proof}

\begin{proof}[Proof of Lemma \ref{lemma:pestov}] The starting point are the structure equations of the surface
\begin{align*}
[V,X]&=-X_{\perp},\\
[V,X_{\perp}]&=X,\\
[X,X_{\perp}]&=-KV.
\end{align*}

The way to introduce the connection and the Higgs field is to replace the operators $X$ by $X+A+\Phi$ and $X_{\perp}$ by $X_{\perp}-V(A)=X_{\perp}+\star A$. Here we understand as always
that $A$ and $\Phi$ act on functions $u:SM\to\C^n$ by multiplication.

Thus we compute
\begin{align*}
&[V,X+A+\Phi]=-X_{\perp}-\star A,\\
&[V,X_{\perp}+\star A]=(X+A+\Phi)-\Phi,\\
&[X+A+\Phi,X_{\perp}+\star A]=-KV-\star F_{A}-\star d_{A}\Phi.
\end{align*}
For the last equation one needs to use (\ref{eq:curv}) and that
$X_{\perp}(\Phi)=\star d\Phi$.

Now we use the first two bracket relations in the following string of equalities
together with the first claim of Lemma \ref{lemma:anti}.

\begin{align*}
-2\langle& (X_{\perp}+\star A)(u),V((X+A+\Phi)(u))\rangle\\
&=-\langle (X_{\perp}+\star A)(u),V((X+A+\Phi)(u))\rangle+\langle V((X_{\perp}+\star A)(u)),(X+A+\Phi)(u)\rangle\\ 
&=\langle (X_{\perp}+\star A)(u),(X_{\perp}+\star A)(u)-(X+A+\Phi)(Vu)\rangle\\
&\;\;\;\;\;\;\;\;\;\;\;\;\;-\langle -(X+A+\Phi)(u)+\Phi u-(X_{\perp}+\star A)(Vu),(X+A+\Phi)(u)\rangle\\
&=|(X_{\perp}+\star A)(u)|^{2}+|(X+A+\Phi)(u)|^2-\langle (X_{\perp}+\star A)(u),(X+A+\Phi)(Vu)\rangle\\
&\;\;\;\;\;\;\;\;\;\;\;\;\;\;+\langle (X_{\perp}+\star A)(Vu),(X+A+\Phi)(u)\rangle-\langle \Phi u,(X+A+\Phi)u\rangle.
\end{align*}
So far we have not used that $u$ vanishes on $\partial(SM)$. We will use it now via the second claim in Lemma \ref{lemma:anti} together with the third bracket relation involving curvatures. Note that $V(u)|_{\partial(SM)}=0$.
We have
\begin{align*}
&-\langle (X_{\perp}+\star A)(u),(X+A+\Phi)(Vu)\rangle=\langle(X+A+\Phi)(X_{\perp}+\star A)(u),V(u)\rangle\\
&=-\langle KV(u)+\star F_{A}u+(\star d_{A}\Phi)u,V(u)\rangle+\langle (X_{\perp}+\star A)(X+A+\Phi)(u),V(u)\rangle\\
&=-\langle KV(u)+\star F_{A}u+(\star d_{A}\Phi)u,V(u)\rangle-\langle (X+A+\Phi)(u),(X_{\perp}+\star A)V(u)\rangle\\
&=-\langle KV(u)+\star F_{A}u+(\star d_{A}\Phi)u,V(u)\rangle-\overline{\langle (X_{\perp}+\star A)V(u),(X+A+\Phi)(u)\rangle}.\\
\end{align*}
Combining we obtain:

\begin{align}\label{eq:pes}
-&2\langle (X_{\perp}+\star A)(u),V((X+A+\Phi)(u))\rangle\\ \notag
&=|(X_{\perp}+\star A)(u)|^{2}+|(X+A+\Phi)(u)|^2-\langle KV(u),V(u)\rangle-\langle\star F_{A}u,V(u)\rangle\\
&-\langle (\star d_{A}\Phi)u,Vu\rangle-\langle \Phi u,(X+A+\Phi)(u)\rangle+2i\Im\langle (X_{\perp}+\star A)(Vu),(X+A+\Phi)(u)\rangle.\notag
\end{align}

To obtain the identity in the lemma we need one more step.
We use the bracket relation $[V,X+A+\Phi]=-X_{\perp}-\star A$ to derive

\begin{align*}
&|(X+A+\Phi)(Vu)|^2\\
&=2\Re\langle (X_{\perp}+\star A)(u),V((X+A+\Phi)(u))\rangle+|V((X+A+\Phi)u)|^2+|(X_{\perp}+\star A)(u)|^2,
\end{align*}
which combined with the real part of (\ref{eq:pes}) above proves the lemma.
\end{proof}

\begin{Remark}{\rm The same Energy identity holds true for {\it closed} surfaces.}
\end{Remark}

To use the Energy identity we need to control the signs of various terms. 
The first easy observation is the following:

\begin{Lemma} Assume $(X+A+\Phi)(u)=F(x)+\alpha_{j}(x)v^j$, where $F:M\to\C^n$
is a smooth function and $\alpha$ is a $\C^n$-valued 1-form. Then
\[|V[(X+A+\Phi)(u))]|^2-|(X+A+\Phi)(u)|^2=-|F|^2\leq 0.\]
\label{lemma:easy}
\end{Lemma}

\begin{proof} It suffices to note the identities:
\[|V[(X+A+\Phi)(u))]|^2=|V(\alpha)|^2=|\alpha|^2,\]
\[|F+\alpha|^2=|\alpha|^2+|F|^2.\]
\end{proof}

Next we have the following lemma due to the absence of conjugate points
on simple surfaces (compare with \cite[Theorem 4.4]{DP}):

\begin{Lemma} Let $M$ be a compact simple surface.
If $u:SM\to\C^n$ is a smooth function
such that $u|_{\partial(SM)}=0$, then
\[|(X+A+\Phi)(Vu)|^2-\langle K\,V(u),V(u)\rangle\geq 0.\]
\label{lemma:nonconj}
\end{Lemma}

\begin{proof} Consider a smooth function $a:SM\to \mathbb R$ which solves the
Riccati equation $X(a)+a^2+K=0$. These exist by the absence of conjugate points (cf. for example \cite[Theorem 6.2.1]{S} or proof of Lemma 4.1 in \cite{SU}). Set for simplicity $\psi=V(u)$. Clearly
$\psi|_{\partial(SM)}=0$.

Let us compute using that $A+\Phi$ is skew-Hermitian:

\begin{align*}
\|(X+A+\Phi)&(\psi)-a\psi\|^2_{\C^n}\\
&=\|(X+A+\Phi)(\psi)\|^2_{\C^n}-2\Re\langle (X+A+\Phi)(\psi),a\psi\rangle_{\C^n}+a^2\|\psi\|_{\C^{n}}^2\\
&=\|(X+A+\Phi)(\psi)\|^2_{\C^{n}}-2a\Re\langle X(\psi),\psi\rangle_{\C^{n}}+a^2\|\psi\|^2_{\C^{n}}.
\end{align*}

Using the Riccati equation we have
\[X(a\|\psi\|^2)=(-a^2-K)\|\psi\|^2+2a\Re\langle X(\psi),\psi\rangle_{\C^{n}}\]
thus
\[\|(X+A+\Phi)(\psi)-a\psi\|^2_{\C^n}=\|(X+A+\Phi)(\psi)\|^2_{\C^n}-K\|\psi\|_{\C^n}^2-X(a\|\psi\|_{\C^n}^2).\]
Integrating this equality with respect to $d\Sigma^3$ and using that $\psi$ vanishes on $\partial(SM)$ we obtain
\[|(X+A+\Phi)(\psi)|^2-\langle K\,\psi,\psi\rangle=|(X+A+\Phi)(\psi)-a\psi|^2\geq 0.\]
\end{proof}

For the remainder of this section we will assume that $\Phi\equiv 0$. This will expose the underlying argument for injectivity of $I_{A}$ more clearly. The Higgs field will be reinstated in the next section and a more involved argument will be needed to deal with it.

The next theorem is one of the key
technical results. Given a smooth function $f:SM\to\C^n$ we consider its Fourier expansion
\[f(x,v)=\sum_{k=-\infty}^{\infty}f_{k},\]
where $f_{k}\in\Omega_k$.

\begin{Theorem} Let $f:SM\to\C^n$ be a smooth function.
Suppose $u:SM\to\C^n$ satisfies
\[\left\{\begin{array}{ll}
X(u)+Au=-f,\\
u|_{\partial(SM)}=0.\\
\end{array}\right.\]
Then if $f_k=0$ for all $k\leq -2$ and $i\star F_{A}(x)$ is a negative definite Hermitian matrix for all $x\in M$, the function $u$ must be holomorphic.
Moreover, if $f_{k}=0$ for all $k\geq 2$ and $i\star F_{A}(x)$ is a positive definite Hermitian matrix for all $x\in M$, the function $u$ must be antiholomorphic.
\label{thm:pi}
\end{Theorem}

\begin{proof} Let us assume that $f_{k}=0$ for $k\leq -2$ and $i\star F_{A}$
is a negative definite Hermitian matrix; the proof of the other claim is
similar.

We need to show that $(\id-i\H)u$ only depends on $x$. We apply $X+A$ to it and use Lemma \ref{lemma:bracket} together with $(\id-i\H)f=f_0+2f_{-1}$ to derive:
\begin{align*}
(X+A)[(\id-i\H)u]&=-f-i(X+A)(\H u)\\
&=-f-i(\H((X+A)(u))-(\X+\star A)(u_{0})-\{(\X+\star A)(u)\}_{0})\\
&=-(\id-i\H)(f)+i(\X+\star A)(u_{0})+i\{(\X+\star A)(u)\}_{0}\\
&=-f_0-2f_{-1}+i(\X+\star A)(u_{0})+i\{(\X+\star A)(u)\}_{0}\\
&=F(x)+\alpha_{x}(v),\\
\end{align*}
where $F:M\to \C^n$ and $\alpha$ is a $\C^n$-valued 1-form.
Now we are in good shape to use the Energy identity from Lemma \ref{lemma:pestov}. We will apply it to
$v=(\id-i\H)u=u_0+2\sum_{k=-\infty}^{-1} u_{k}$.
We know from Lemma \ref{lemma:easy} that its right hand side is $\leq 0$ and using Lemma \ref{lemma:nonconj}
we deduce
\[\langle \star F_{A}v,V(v)\rangle \geq 0.\]
But on the other hand
\[\langle \star F_{A}v,V(v)\rangle=-4\sum_{k=-\infty}^{-1}k\langle i\star F_{A}u_{k},u_{k}\rangle\]
and since $i\star F_{A}$ is negative definite this forces $u_{k}=0$ for all $k<0$.

\end{proof}

We are now ready to prove the main result on injectivity of $I_{A}$.

\begin{Theorem} Let $M$ be a compact simple surface. Assume that $f:SM\to\C^n$ is a smooth function of the form
$F(x)+\alpha_{j}(x)v^j$, where $F:M\to\C^n$ is a smooth function
and $\alpha$ is a $\C^n$-valued 1-form. If $I_{A}(f)=0$, then
$F\equiv 0$ and $\alpha=d_{A}p$, where $p:M\to\C^n$ is a smooth
function with $p|_{\partial M}=0$.
\label{thm:injective}
\end{Theorem}

\begin{proof}
Consider the area form $\omega_g$ of the metric $g$. Since $M$ is a disk
there exists a smooth 1-form $\varphi$ such that $\omega_{g}=d\varphi$.
Given $s\in\mathbb R$, consider the Hermitian connection
\[A_{s}:=A-is\varphi\,\id.\]
Clearly its curvature is given by
\[F_{A_{s}}=F_{A}-is\omega_{g}\id\]
therefore
\[i\star F_{A_{s}}=i\star F_{A}+s\id,\]
from which we see that there exists $s_0>0$ such that
 for $s>s_0$, $i\star F_{A_{s}}$ is positive definite
and for $s<-s_0$, $i\star F_{A_{s}}$ is negative definite.

Since $I_{A}(f)=0$, Proposition \ref{prop:smooth} implies that
there is a {\it smooth} $u:SM\to\C^n$ such that $(X+A)(u)=-f$ and
$u|_{\partial(SM)}=0$ (to abbreviate the notation we write $u$ instead of
$u^f$).

Let $e^{sw}$ be an integrating factor of $-is\varphi$. In other words
$w:SM\to\C$ satisfies $X(w)=i\varphi$. By Theorem \ref{prop_holomorphic_integrating_factors} we know we can choose 
$w$ to be holomorphic or antiholomorphic, this will be crucial below.
Observe now that $u_{s}:=e^{sw}u$ satisfies $u_{s}|_{\partial(SM)}=0$ and
solves
\[(X+A_{s})(u_{s})=-e^{sw}f.\]

Choose $w$ to be holomorphic. Since $f=F(x)+\alpha_{j}(x)v^j$, the function
$e^{sw}f$ has the property that its Fourier coefficients $(e^{sw}f)_{k}$
vanish for $k\leq -2$.
Choose $s$ such that $s<-s_0$ so that
$i\star F_{A_{s}}$ is negative definite. Then Theorem \ref{thm:pi} implies
that $u_s$ is holomorphic and thus $u=e^{-sw}u_{s}$ is also holomorphic.

Choosing $w$ antiholomorphic and $s>s_0$ we show similarly that $u$ is
antiholomorphic. This implies that $u=u_0$ which together with
$(X+A)u=-f$, gives $d_{A}u_{0}=-f$. If we set $p=-u_0$ we see right away
that $F\equiv 0$ and $\alpha=d_{A}p$ as desired.
\end{proof}

\section{Injectivity in the presence of a Higgs field} \label{sec:injective_higgs}

We now extend the injectivity result in the previous section to the case where a Higgs field is included. This proves Theorem \ref{thm:injective_higgs} from the Introduction, restated here:

\begin{Theorem} Let $M$ be a compact simple surface. Assume that $f:SM\to\C^n$ is a smooth function of the form
$F(x)+\alpha_{j}(x)v^j$, where $F:M\to\C^n$ is a smooth function
and $\alpha$ is a $\C^n$-valued 1-form. Let also $A: SM \to \mathfrak{u}(n)$ be a unitary connection and $\Phi: M \to \mathfrak{u}(n)$ a skew-Hermitian matrix function. If $I_{A,\Phi}(f)=0$, then
$F = \Phi p$ and $\alpha=d_{A}p$, where $p:M\to\C^n$ is a smooth
function with $p|_{\partial M}=0$.
\label{thm:injective_higgs2}
\end{Theorem}
\begin{proof}
Since $I_{A,\Phi}(f) = 0$, there is a function $u: SM \to \C^n$ with $u|_{\partial (SM)} = 0$ satisfying 
$$
(X+A+\Phi)u = -f.
$$
It follows from Proposition \ref{prop:smooth} that $u \in C^{\infty}(SM,\C^n)$. We will prove that $u$ is both holomorphic and antiholomorphic. If this is the case then $u=u_0$ only depends on $x$ and $u_0|_{\partial M} = 0$, and we have 
$$
du_0 + Au_0 = -\alpha, \quad \Phi u_0 = -F.
$$
The theorem follows upon choosing $p = -u_0$.

It remains to show that $u$ is holomorphic (the antiholomorphic case follows by an analogous argument). The first step, as in the proof of Theorem \ref{thm:injective}, is to replace $A$ by a connection whose curvature has a definite sign. We choose a real valued $1$-form $\varphi$ such that $d\varphi = \omega_g$, and let 
$$
A_s := A + is\varphi \id.
$$
Here $s > 0$ so that $A_s$ is Hermitian and $i \star F_{A_s} = i \star F_A - s\id$. We use Theorem \ref{prop_holomorphic_integrating_factors} to find a holomorphic scalar function $w \in C^{\infty}(SM)$ satisfying $Xw = -i\varphi$. Then $u_s = e^{sw} u$ satisfies 
$$
(X + A_s + \Phi)u_s = -e^{sw} f.
$$
Let $v = (\id-i\hilb)u_s - (u_s)_0$. The commutator formula in Lemma \ref{lemma:bracket} shows that 
\begin{multline*}
(X+A_s+\Phi)v = (\id-i\hilb)((X+A_s+\Phi)u_s) + i[\hilb,X+A_s]u_s - (X+A_s+\Phi)(u_s)_0 \\
 = -(\id-i\hilb)(e^{sw}f) + i(X_{\perp} + \star A_s)(u_s)_0 + i\left\{ (X_{\perp} + \star A_s) u_s \right\}_0 - (X+A_s+\Phi)(u_s)_0.
\end{multline*}
Here $(\id-i\hilb)(e^{sw} f)$ is the sum of a $0$-form and a $1$-form since $w$ is holomorphic. Consequently 
$$
(X+A_s+\Phi)v = -h(x) - \beta_j(x)v^j
$$
where $h \in C^{\infty}(M,\C^n)$ and $\beta$ is a $1$-form.

Applying the Energy identity in Lemma \ref{lemma:pestov} with connection $A_s$ to the function $v$, which also satisfies $v|_{\partial(SM)} = 0$, we see that 
\begin{multline}
|(X+A_s+\Phi)(Vv)|^2 - \langle K\,V(v),V(v)\rangle + |(X+A_s+\Phi)v|^2 - |V[(X+A_s+\Phi)v]|^2  \\
 - \langle \star F_{A_s}v,V(v)\rangle - \Re\langle (\star d_{A_s}\Phi) v,V(v)\rangle - \Re\langle\Phi v,(X+A_s+\Phi)v\rangle = 0. \label{pestov_higgs_identity}
\end{multline}
It was proved in Lemmas \ref{lemma:easy} and \ref{lemma:nonconj} that 
\begin{gather}
|(X+A_s+\Phi)(Vv)|^2 - \langle K\,V(v),V(v)\rangle \geq 0, \label{pestov_higgs_intermediate1}\\
|(X+A_s+\Phi)v|^2 - |V[(X+A_s+\Phi)v]|^2 = |h|^2 \geq 0. \label{pestov_higgs_intermediate2}
\end{gather}
The term involving the curvature of $A_s$ satisfies 
\begin{equation} \label{pestov_higgs_intermediate3}
-\langle \star F_{A_s}v,V(v)\rangle = \sum_{k=-\infty}^{-1} |k| \langle -i\star F_{A_s}v_{k},v_{k}\rangle \geq (s-\lVert F_A \rVert_{L^{\infty}(M)}) \sum_{k=-\infty}^{-1} |k| |v_k|^2.
\end{equation}
Here we can choose $s > 0$ large to obtain a positive term. For the next term in \eqref{pestov_higgs_identity}, we consider the Fourier expansion of $d_{A_s} \Phi = d_A \Phi = a_1 + a_{-1}$ where $a_{\pm 1}\in\Omega_{\pm 1}$.
Note that $\star d_{A}\Phi=-V(d_{A}\Phi)=-ia_{1}+ia_{-1}$. Then, since $v_k = 0$ for $k \geq 0$, 
\begin{align*}
\langle (\star d_A \Phi) v,V(v)\rangle &= \sum_{k=-\infty}^{-1} \langle -ia_1 v_{k-1} + ia_{-1} v_{k+1}),  ik v_k \rangle \\
 &=  \sum_{k=-\infty}^{-1} |k| \left[ \langle a_1 v_{k-1}, v_k \rangle - \langle a_{-1} v_{k+1}, v_k \rangle \right].
\end{align*}
Consequently 
\begin{equation} \label{pestov_higgs_intermediate4}
\Re\langle (\star d_A \Phi) v,V(v)\rangle \leq C_{A,\Phi} \sum_{k=-\infty}^{-1} |k| |v_k|^2.
\end{equation}

The last term in \eqref{pestov_higgs_identity} requires the most work. We note that $v_k=0$ for $k \geq 0$ and that $(X+A_s+\Phi)v$ is the sum of a $0$-form and a $1$-form. Therefore 
\begin{align*}
\langle\Phi v,(X+A_s+\Phi)v\rangle = \langle\Phi v_{-1},((X+A_s+\Phi)v)_{-1}\rangle.
\end{align*}
Recall from Section \ref{sec:prelim} that we may write $X = \eta_+ + \eta_-$ where $\eta_+ = (X+iX_{\perp})/2: \Omega_k \to \Omega_{k+1}$ and $\eta_- = (X-iX_{\perp})/2: \Omega_k \to \Omega_{k-1}$. Expand $A=A_{1}+A_{-1}$ and
$\varphi=\varphi_{1}+\varphi_{-1}$ so that $A_{s}=(A_{1}+is\varphi_{1}\id)+
(A_{-1}+si\varphi_{-1}\id):= a_1 + a_{-1}$ where $a_j \in \Omega_j$.
Since $A_s$ is Hermitian we have $a_{\pm 1}^* = -a_{\mp 1}$.

The equation $(X+A_s+\Phi)v = -h(x) - \beta_j(x)v^j$ implies that 
\begin{gather*}
\eta_+ v_{-2} + a_1 v_{-2} + \Phi v_{-1} = ((X+A_s+\Phi)v)_{-1}, \\
\eta_+ v_{-k-1} + a_1 v_{-k-1} + \eta_- v_{-k+1} + a_{-1} v_{-k+1} + \Phi v_{-k} = 0, \qquad k \geq 2.
\end{gather*}
Note that $\langle \eta_{\pm} a,b \rangle = -\langle a, \eta_{\mp} b \rangle$ when $a|_{\partial(SM)} = 0$; this is proved exactly as in Lemma \ref{lemma:anti}. Using this and that $\Phi$ is skew-Hermitian, we have 
\begin{align*}
 &\Re \langle\Phi v_{-1},((X+A_s+\Phi)v)_{-1}\rangle = \Re \langle\Phi v_{-1},\eta_+ v_{-2} + a_1 v_{-2} + \Phi v_{-1} \rangle \\
 &= \Re \left[ \langle \eta_{-} v_{-1}, \Phi v_{-2} \rangle - \langle (\eta_- \Phi) v_{-1}, v_{-2} \rangle + \langle \Phi v_{-1}, a_1 v_{-2} \rangle + |\Phi v_{-1}|^2 \right].
\end{align*}
We claim that for any $N \geq 1$ one has 
$$
\Re \langle\Phi v_{-1},((X+A_s+\Phi)v)_{-1}\rangle = p_N + q_N
$$
where 
\begin{gather*}
p_N := (-1)^{N-1} \Re \langle \eta_{-} v_{-N}, \Phi v_{-N-1} \rangle, \\
q_N := \Re \sum_{j=1}^N \big[ (-1)^j \langle (\eta_- \Phi) v_{-j}, v_{-j-1} \rangle + (-1)^{j-1} \langle \Phi v_{-j}, a_1 v_{-j-1} \rangle + (-1)^{j-1} |\Phi v_{-j}|^2 \big] \\
 + \Re \sum_{j=1}^{N-1} (-1)^j \langle a_{-1} v_{-j}, \Phi v_{-j-1} \rangle.
\end{gather*}
We have proved the claim when $N=1$. If $N \geq 1$ we compute 
\begin{align*}
p_N &= (-1)^N \Re \langle (\eta_+ + a_1) v_{-N-2} + a_{-1} v_{-N} + \Phi v_{-N-1}, \Phi v_{-N-1} \rangle \\
 &= (-1)^N \Re \big[ \langle \Phi v_{-N-2}, \eta_- v_{-N-1} \rangle - \langle v_{-N-2}, (\eta_- \Phi) v_{-N-1} \rangle \\
 &\quad \quad \quad \quad \quad \quad + \langle a_1 v_{-N-2} + a_{-1} v_{-N} + \Phi v_{-N-1}, \Phi v_{-N-1} \rangle \big] \\
 &= p_{N+1} + q_{N+1} - q_N.
\end{align*}
This proves the claim for any $N$.

Note that since $|\eta_- v|^2 = \sum |\eta_- v_k|^2$, we have $\eta_- v_k \to 0$ and similarly $v_k \to 0$ in $L^2(SM)$ as $k \to -\infty$. Therefore $p_N \to 0$ as $N \to \infty$. We also have 
$$
|q_N| \leq C_{\Phi} \sum |v_k|^2 + \left\lvert \sum_{j=1}^N (-1)^j \langle [a_{-1},\Phi] v_{-j}, v_{-j-1} \rangle \right\rvert \leq C_{A,\Phi} \sum |v_k|^2.
$$
Here it was important that the term in $a_{-1}$ involving $s$ is a scalar, so it goes away when taking the commutator $[a_{-1},\Phi]$. After taking a subsequence, $(q_N)$ converges to some $q$ having a similar bound. We finally obtain 
\begin{equation} \label{pestov_higgs_intermediate5}
\Re\langle\Phi v,(X+A_s+\Phi)v\rangle = \lim_{N \to \infty} (p_N + q_N) \leq C_{A,\Phi} \sum |v_k|^2.
\end{equation}

Collecting the estimates \eqref{pestov_higgs_intermediate1}--\eqref{pestov_higgs_intermediate5} and using them in \eqref{pestov_higgs_identity} shows that 
$$
0 \geq |h|^2 + (s-C_{A,\Phi}) \sum_{k=-\infty}^{-1} |k| |v_k|^2.
$$
Choosing $s$ large enough implies $v_k = 0$ for all $k$. This proves that $u_s$ is holomorphic, and therefore $u = e^{-sw} u_s$ is holomorphic as required.
\end{proof}

\section{Scattering data determines the connection and the Higgs field} \label{sec:scattering_relation}

Given a Hermitian connection $A$ and skew-Hermitian Higgs field $\Phi$, let $C^{A,\Phi}_{+}:\partial_{+}(SM)\to U(n)$
denote the scattering data introduced in Section \ref{scatt}. (Recall that $C^{A,\Phi}_{+}$ and $C^{A,\Phi}_{-}$ are equivalent information once the scattering relation $\alpha$ for $(M,g)$ is known.) For brevity we omit the subscript $``+"$ in $C_{+}^{A,\Phi}$ from now on.

In this section we prove Theorem \ref{thm:inverse} from the Introduction. We restate it in the following form:

\begin{Theorem} Assume $M$ is a compact simple surface, let
$A$ and $B$ be two Hermitian connections, and let $\Phi$ and $\Psi$ be two skew-Hermitian matrix functions.
Then $C^{A,\Phi}=C^{B,\Psi}$ implies that there exists a smooth
$U:M\to U(n)$ such that $U|_{\partial M}=\id$ and
$B=U^{-1}dU+U^{-1}AU$, $\Psi = U^{-1} \Phi U$.
\end{Theorem}

\begin{proof} Consider the unique matrix solution
$U_{A,\Phi}:SM\to U(n)$ to the transport equation:
\[\left\{\begin{array}{rl}
(X+A+\Phi)U_{A,\Phi} &\!\!\!=0,\\
U_{A,\Phi}|_{\partial_{-}(SM)} &\!\!\!=\id.\\
\end{array}\right.\]
Then $C^{A,\Phi} = U_{A,\Phi}|_{\partial_{+}(SM)}$. Similarly consider 
the unique matrix solution
$U_{B,\Psi}:SM\to U(n)$ to the transport equation:
\[\left\{\begin{array}{rl}
(X+B+\Psi)U_{B,\Psi} &\!\!\! =0,\\
U_{B,\Psi}|_{\partial_{-}(SM)} &\!\!\! =\id.\\
\end{array}\right.\]
Then $C^{B,\Psi} = U_{B,\Psi}|_{\partial_{+}(SM)}$.

Define $U:=U_{A,\Phi}(U_{B,\Psi})^{-1}:SM\to U(n)$. One easily checks that
$U$ satisfies:
\[\left\{\begin{array}{rl}
XU + AU + \Phi U - UB - U\Psi &\!\!\!= 0,\\
U|_{\partial(SM)} &\!\!\! =\id.\\
\end{array}\right.\]
We are going to show that $U$ is in fact smooth and it only depends
on the base point $x\in M$ thus giving rise to the desired function
$U:M\to U(n)$.

Consider $W:=U-\id:SM\to \C^{n \times n}$, where $\C^{n \times n}$ stands for the set of all
$n\times n$ complex matrices. Clearly $W$ satisfies
\begin{align}
&XW+AW-WB+\Phi W - W \Psi = B - A + \Psi - \Phi, \,\label{eq:uno}\\
&W|_{\partial(SM)}=0.\label{eq:dos}
\end{align}

We now interpret these last two equations in terms of a suitable attenuated
ray transform. We introduce a new connection $\hat{A}$ on the trivial bundle $M\times \C^{n \times n}$ and a new Higgs field $\hat{\Phi}$ as follows: given a matrix $R \in \C^{n \times n}$ we define
$\hat{A}(R):=AR-RB$ and $\hat{\Phi}(R) := \Phi R - R \Psi$. One easily checks that $\hat{A}$ is Hermitian
if $A$ and $B$ are, and also that $\hat{\Phi}$ is skew-Hermitian when $\Phi$ and $\Psi$ are. Then equations (\ref{eq:uno}) and (\ref{eq:dos})
are saying precisely that $I_{\hat{A},\hat{\Phi}}(A-B+\Phi-\Psi)=0$. Note that there
is a unique solution $W:SM\to \C^{n \times n}$ satisfying (\ref{eq:uno}) 
and $W|_{\partial_{-}(SM)}=0$.

Up to this point the argument is valid on any nontrapping compact manifold with strictly convex boundary. We now use that fact that $(M,g)$ is simple, and apply Theorem \ref{thm:injective_higgs} to conclude that there is a smooth function $p:M\to \C^{n \times n}$
such that $p$ vanishes on $\partial M$ and
\[ Xp + Ap - pB + \Phi p - p \Psi = B-A + \Psi - \Phi.\]
Hence $W=p$ and $U=W+\id$ has all the desired properties.
\end{proof}

\end{document}